\documentclass[11pt, twoside]{article}
\usepackage{amsfonts}
\usepackage{ulem}

\usepackage{amssymb}
\usepackage{amsmath}
\usepackage{amsthm}
\usepackage{xcolor}
\usepackage{mathrsfs}
\usepackage{mathabx}

\usepackage{verbatim}

\allowdisplaybreaks

\allowdisplaybreaks

\def\R{{\mathbb R}}

\def\R{{\mathbb R}}

\def\z+{{\mathbb Z}_+}

\def\bint{{\ifinner\rlap{\bf\kern.30em--}
\int\else\rlap{\bf\kern.35em--}\int\fi}\ignorespaces}

\def\sbint{{\ifinner\rlap{\bf\kern.32em--}
\hspace{0.078cm}\int\else\rlap{\bf\kern.45em--}\int\fi}\ignorespaces}

\newtheorem{theorem}{Theorem}[section]
\newtheorem{lemma}[theorem]{Lemma}

\newtheorem{proposition}[theorem]{Proposition}
\theoremstyle{definition}
\newtheorem{example}[theorem]{Example}

\newtheorem{definition}[theorem]{Definition}
\numberwithin{equation}{section}

\numberwithin{equation}{section}

\numberwithin{equation}{section}

\begin{document}

\arraycolsep=1pt

\title{\Large\bf Approximation via partial Hausdorff integrals on $H^1(\R)$ \footnotetext{\hspace{-0.35cm} {\it 2020
Mathematics Subject Classification}. {41A25, 42A38.}
\endgraf{\it Key words and phrases.} Hausdorﬀ operators, approximation, Hardy space.
\endgraf This project is supported by Natural Science Foundation of Xinjiang Province (No.2024D01C40) and National Natural Science Foundation of China (No.12261083).
\endgraf $^\ast$\,Corresponding author.
}}
\author{Zifei Yu and Baode Li$^\ast$}
\date{ }
\maketitle

\vspace{-0.8cm}

\begin{center}
\begin{minipage}{13cm}\small
{\noindent{\bf Abstract.}
We obtain the result of approximating \( f \) in the \( H^1(\mathbb{R}) \) norm using partial Hausdorff integrals. Specifically, by leveraging the homogeneous multiplier theory of \( H^1(\mathbb{R}) \) and the \( K \) functional theory, one result from Pinos and Liflyand [CMB,~2021,~64,~no.3] is extended from \( L^p(\mathbb{R}) \) ( \( 1 \leq p \leq \infty \)) to \( H^1(\mathbb{R}) \). As applications, four examples of partial Hausdorff integrals are also given.} 
\end{minipage}
\end{center}

\section{Introduction\label{s1}}

Let $\varphi$ be a measurable real-valued function on $\R$. Assume that $a:\R\to\R$, $a$ is measurable, $a\neq0$ a.e., and for every set $E$ of zero Lebesgue measure, the set $a^{-1}(E)$ also has zero measure. Under these assumptions, Burenkov and Liflyand \cite[Theorem 1]{Lif2019} proved that the function \((x,t) \mapsto f(a(t)x)\) is measurable on $\R^2$. Let $\varphi|a|^\frac{1}{2}\in L^1(\R)$. A general {\it Hausdorff operator} $H$ of $f\in L^2(\R)$ is defined by
\begin{equation}\label{eq23}
H(f)(x):=H_{\varphi,a}(f)(x)=\int_{\R} \varphi(t)|a(t)|f(a(t)x)dt,~x\in \R.
\end{equation}
By Minkowski's inequality and substituting $a(t)x=\tilde{x}$, it is easy to verify that $H$ is bounded on $L^2(\R)$:
\[
\|Hf\|_{L^2}
\leq \int_{\R} |\varphi(t)||a(t)| \left(\int_{\R}|f(a(t)x)|^2 dx\right)^\frac{1}{2}dt=\int_{\R}|\varphi(t)||a(t)|^\frac{1}{2}dt\|f\|_{L^2}.
\]

Taking $a(t)=1/t$ when $t\neq0$ and $a(0)=0$, then $H$ is a {\it one-dimensional Hausdorff operator} $H_\varphi$, i.e.,
\[
H_{\varphi}(f)(x):=\int_{\R} \frac{\varphi(t)}{|t|} f\left(\frac{x}{t}\right) dt.
\]

Suppose $a$ is additionally odd and such that $|a|$ is decreasing, positive and bijective on $(0,+\infty)$. Pinos and Liflyand \cite{Liflyand2021} defined 
\begin{align*}
({H_N\widehat{f}})\widecheck{~}(x)
&:=
\frac{1}{2\pi}\int_{-N}^{N} H\widehat{f}(u)e^{iux}du
\\
&=\frac{1}{2\pi}\int_{-N}^{N} \int_{\R}\varphi(t)|a(t)|\int_{\R}f(s)e^{-ia(t)su}dsdte^{iux}du
\\
&=\frac{1}{\pi}\int_{\R}\varphi(t)|a(t)|\int_{\R}f(s)\frac{\sin N(x-a(t)s)}{x-a(t)s}ds dt, \quad N>0.
\end{align*}
By substituting $\tilde{s}=N(\frac{x}{a(t)}-s)$,
\begin{equation}\label{eq25}
(H_N\widehat{f})\widecheck{~}(x)=\frac{1}{\pi}\int_{\R}\varphi(t)\int_{\R}f\left(\frac{x}{a(t)}-\frac{s}{N}\right)\frac{\sin|a(t)|s}{s}dsdt.
\end{equation}
Suppose that $\varphi\in L^1(\R)$ and $\int_{\R}\varphi(t)dt=1$. To approximate \( f \), Pinos and Liflyand \cite{Liflyand2021} defined the {\it  partial Hausdorff integrals} by
\begin{align}
F_N(x):&=\left(H_N\widehat{f(x+\cdot)}\right)\widecheck{~}(0)=\left(H_N\widehat{\tau_xf}\right)\widecheck{~}(0)\nonumber\\
&=\frac{1}{\pi}\int_{\R}\varphi(t)\int_{\R}f\left(x-\frac{s}{N}\right)\frac{\sin(|a(t)|s)}{s} {d}s dt.\label{eq24}
\end{align}
In \cite{Liflyand2021}, Pinos and Liflyand obtained the approximation of \( f \) by \( F_N \) in the \( L^p(\R) \)-norm using the {\it \( L^p(\R) \)-modulus of continuity} when $1\leq p\leq\infty$, $\varphi\max\{|a|^\frac{1}{p},|a|^\frac{1}{2}\}\in L^1(\R^n)$ and $f\in L^1\cap L^p(\R^n)$.

In our proof, for convenience, we set \( 1/N = \epsilon \) in $(\ref{eq24})$ and denote
\begin{equation}\label{eq11}
F_\epsilon(x)
:=
\frac{1}{\pi}\int_{\R}\varphi(t)\int_{\R}f(x-\epsilon s)\frac{\sin(|a(t)|s)}{s} ds dt. 
\end{equation}
We obtained the approximation of \( f \) by \( F_\epsilon \) in the \( H^1(\R) \)-norm using the {\it K functional} from \cite[Chapter 4]{Lu1995}.

The structure of this paper is as follows: In Section \ref{s2}, we introduce the necessary definitions and notations; in Section \ref{s3}, we first prove that the partial Hausdorff integrals are uniformly bounded on \( H^1(\R) \), and then we show that \( f \) can be approximated by the partial Hausdorff integrals under the \( H^1(\R) \)-norm, along with several necessary lemmas; and in Section \ref{s4}, we provide four examples.

The proof of  \cite[Theorem 2.2]{Lif2019} uses the modulus of continuity of \( L^p(\mathbb{R}^n) \) (\( 1 \leq p \leq \infty \)) to show that partial Hausdorff integrals can approximate \( f \in L^p(\mathbb{R}^n) \) under the \( L^p(\mathbb{R}^n) \)-norm. Different from this approach, we utilize the $K$ functional to prove that, under the same conditions, partial Hausdorff integrals can approximate \( f \in H^1(\mathbb{R}) \) under the \( H^1(\mathbb{R}) \)-norm.




\section{Notations and definitions \label{s2}}

\begin{definition}\label{def2.2}\cite{Loukas2024}
For \(1\leq p<\infty\), we define the \(L^p(\R)\) norm of a measurable function \(f\) by
\[
\|f\|_{L^p(\R)}:=\left(\int_{\R}|f(x)|^p d x\right)^{\frac{1}{p}}.
\]
For any \(1\leq p< \infty\), we define \(L^p(\R)\) to be the space of all measurable functions \(f\) with \(\|f\|_{L^p(\R)}<\infty\). 

\end{definition}

\begin{definition}\cite{Loukas2024}
Let \( f \in L^2(\mathbb{R}) \). The {\it Fourier transform} \( \widehat{f} \) of \( f \) is defined by
\[
\widehat{f}(x)=\int_{\R} f(x) e^{-ixy} dy\overset{L^2(\R)}{:=}\lim_{R\to+\infty}\int_{|x|\leq R} f(x)e^{-ixy}dy.
\]
Analogously define the {\it inverse Fourier transform} \( f^\vee \) of $f$ by
\[
f^\vee(x)=\frac{1}{2\pi}\int_{\R}f(x)e^{ixy}dx\overset{L^2(\R)}{:=}\lim_{R\to+\infty}\frac{1}{2\pi}\int_{|x|\leq R} f(x)e^{ixy}dx.
\]
\end{definition}

\begin{definition}
Suppose \(\Phi\in \mathscr{S}\left(\R\right)\) with $\int_{\R} \Phi(x) d x \neq 0 .$
For any tempered distributions $f\in\mathscr{S}'(\R)$, the {\it radial maximal operator} $M^+_{\Phi}$ of $f$ is defined by
\[
M_{\Phi}^+f(x):=\sup _{0<s<\infty}\left|\Phi_s * f(x)\right|,\quad x\in\R,
\]
where \(\Phi_s(x):=\frac{1}{s} \Phi\left(\frac{x}{s}\right)\).
The {\it Hardy space} \(H^1(\R)\) is the space of all tempered distributions \(f\) satisfying
\[
\|f\|_{H^1}:=\left\|M_{\Phi}^+f\right\|_{L^1(\R)}<\infty .
\]
\end{definition}

\begin{definition}\label{def1}\cite[Page 174]{Lu1995}
Suppose $f\in H^1(\R)$ and $\sigma>0$. The {\it $\sigma$-th order Riesz derivative} $I^\sigma f$ is defined by
\begin{equation*}\label{eq1}
\widehat{I^\sigma f}(x)=|x|^\sigma \widehat{f}(x).
\end{equation*}
Then we define some subspaces $H^{1,\sigma}(\R)$ of $H^1(\R)$ as follows
\begin{equation*}\label{eq2}
H^{1,\sigma}:=\{f\in H^1: I^\sigma f\in H^1\}.
\end{equation*}
\end{definition}

\begin{definition}\label{def2}\cite[Page 174]{Lu1995}
Suppose that $\sigma>0$, $t>0$ and $f\in H^1(\R)$. {\it The $\sigma$-th order $K$ functional of $f$} is defined by
\begin{equation*}
K_\sigma(f,t)_{H^1}:=\inf_{g\in H^{1,\sigma}}\{\|f-g\|_{H^1}+t^\sigma\|I^\sigma g\|_{H^1}\}.
\end{equation*}
\end{definition}

\begin{proposition}\label{pro1}\cite[Page 174]{Lu1995}
Suppose that $\sigma>0$ and $t>0$. If $f\in H^1(\R)$, then
\begin{equation*}\label{eq4}
\lim_{t\to0^+}K_{\sigma}(f,t)_{H^1}=0.
\end{equation*}
\end{proposition}



\begin{definition}\label{def3}\cite[Page 176]{Lu1995}
Suppose that $m\in L^\infty(\R)$. If a family of operators $\{M_\epsilon\}_{\epsilon>0}$ defined by the equality
\begin{equation}\label{eq4}
\widehat{M_{\epsilon} f} (x)=m(\epsilon x)\widehat{f}(x),\quad f\in H^1 \cap L^2(\R)
\end{equation}
can be extended into a family of bounded operators on $H^1(\R)$, and their norms are uniformly bounded in $\epsilon$, the $m(x)$ is called a {\it homogeneous $H^1$ multiplier}.
\end{definition}

\section{Main results \label{s3}}

In this section, we first prove that the partial Hausdorff integrals $\{F_\epsilon\}_{\epsilon>0}$ are uniformly bounded on $H^1(\R)$, then we further obtain the approximation via partial Hausdorff integrals on $H^1(\R)$.

\begin{theorem}\label{th2}
Suppose that $a$ is a measurable function and $\varphi\in L^1(\R)$ satisfies $\int_{\R}\varphi(t)dt=1$ and
\begin{equation}\label{eq1103}
\varphi|a|^\frac{1}{2}\in L^1(\R).
\end{equation}
Then for any $\epsilon>0$ and $f\in H^1(\R)$,
\[
\|F_\epsilon\|_{H^1}\leq C\|f\|_{H^1},
\]
where $C$ is independent of $\epsilon$.
\end{theorem}

To prove Theorem \ref{th2}, we need two lemmas as follows.

\begin{lemma}\label{lem2}\cite[Page 114]{Stein1993}
Suppose \( g \) is a locally integrable function away from the origin on $\R$, and $|\widehat{g}(x)|\leq A_1$, $x\in\R$. Let
\begin{equation}\label{eq6}
(Tf)(x):=f*g(x), ~f\in L^2(\R).
\end{equation}
If 
\begin{equation}\label{eq7}
\int_{|x|\geq2|y|}|g(x-y)-g(x)|dy\leq A_2,~\text{whenever}~y\neq 0,
\end{equation}
then $T$ is bounded on $H^1(\R)$, that is,
\begin{equation}\label{eq8}
\|Tf\|_{H^1}\leq C(A_1,A_2) \|f\|_{H^1},~f\in H^1(\R).
\end{equation}
\end{lemma}

\begin{lemma}\label{lem3}\cite[Chapter 6]{Bracewell1965}
Let $f(x)=\frac{\sin x}{x},x\neq0$. The Fourier transform of $f$ is
\[
h(x):=\widehat{f}(x)
=
\begin{cases}
\pi, &|x|<1;\\
\frac{1}{2}\pi, &|x|=1;\\
0, &|x|>1.
\end{cases}
\]
\end{lemma}

\begin{proof}[Proof of Theorem \ref{th2}]
 Let $f\in L^2(\R)$. By substituting $\epsilon s=\tilde{s}$ and Fubini's theorem, we have
\[
F_\epsilon(x)=\frac{1}{\pi}\int_{\R} f(x-s) \int_{\R} \varphi(t) \frac{\sin(|a(t)|\epsilon^{-1}s)}{s}dtds.
\]
Denote
\[
K_\epsilon (s):=\frac{1}{\pi}\int_{\R} \varphi(t) \frac{\sin(|a(t)|\epsilon^{-1}s)}{s} dt,\quad K(s):=\frac{1}{\pi}\int_{\R} \varphi(t) \frac{\sin(|a(t)|s)}{s} dt.
\]
Then we get
\begin{equation}\label{eq16}
F_\epsilon (x)=K_\epsilon*f(x).
\end{equation}

Now let us prove that $K_\epsilon\in L^2(\R)$. By Minkowskis' inequality, substituting $|a(t)|\epsilon^{-1}s=\tilde{s}$, $\frac{\sin x}{x}$ is even function on $\R$, $|\sin x|\leq 1$, $|\sin x|\leq |x|$  and $(\ref{eq1103})$, we have
\begin{align}
\|K_\epsilon\|_{L^2}
&\leq
\frac{1}{\pi}\int_{\R}|\varphi(t)|\left(\int_{\R}\left|\frac{\sin(|a(t)|\epsilon^{-1}s)}{s}\right|^2ds\right)^\frac{1}{2}dt\nonumber\\
&=\frac{\epsilon^{-\frac{1}{2}}}{\pi}\int_{\R}|\varphi(t)||a(t)|^\frac{1}{2}dt\left(\int_{\R}\left|\frac{\sin s}{s}\right|^2ds\right)^\frac{1}{2}\nonumber\\
&= \frac{\epsilon^{-\frac{1}{2}}}{\pi}\int_{\R}|\varphi(t)||a(t)|^\frac{1}{2}dt\left(2\int_{0}^1\left|\frac{\sin s}{s}\right|^2ds+2\int_1^\infty\left|\frac{\sin s}{s}\right|^2ds\right)^\frac{1}{2}\nonumber\\
&\leq \frac{\epsilon^{-\frac{1}{2}}}{\pi}\int_{\R}|\varphi(t)||a(t)|^\frac{1}{2}dt\left(2\int_{0}^11ds+2\int_1^\infty\frac{1}{s^2}ds\right)^\frac{1}{2}\nonumber\\
&<\infty.
\end{align}
By the properties of the Fourier transform and substituting $\epsilon^{-1}y=\tilde{y}$ , we obtain
\begin{align}
\widehat{F_\epsilon}(x)
&=\widehat{K_\epsilon}(x)\widehat{f}(x)\nonumber\\
&=\widehat{f}(x) \int_{\R} \frac{1}{\pi} \int_{\R} \varphi(t) \frac{\sin(|a(t)|\epsilon^{-1}y)}{y} dt e^{-ixy}dy\nonumber\\
&=\widehat{f}(x) \int_{\R} \frac{1}{\pi} \int_{\R} \varphi(t) \frac{\sin(|a(t)|y)}{y} dt e^{-ix\epsilon y}dy\nonumber\label{eq12011}\\
&=\widehat{K}(\epsilon x)\widehat{f}(x).
\end{align}
By substituting $|a(t)|y=\tilde{y}$, Fubini's theorem and Lemma \ref{lem3}, we have
\begin{align}
\widehat{K}(x)
&=\int_{\R} \frac{1}{\pi} \int_{\R} \varphi(t) \frac{\sin(|a(t)|y)}{y} dt e^{-ixy}dy\nonumber\\
&=\frac{1}{\pi}  \int_{\R}\varphi(t)\int_{\R}\frac{\sin y}{y} e^{-ix|a(t)|^{-1}y}dydt \nonumber\\
&=\frac{1}{\pi} \int_{\R}\varphi(t)\left(\frac{\sin\cdot}{\cdot}\right)\widehat{~}(x|a(t)|^{-1})dt\nonumber 
\\
&=\frac{1}{\pi} \int_{\R} \varphi(t) h(x|a(t)|^{-1}) dt.\label{eq14} 
\end{align}
By $(\ref{eq14})$ and Lemma \ref{lem3}, we obtain
\begin{equation}\label{eq107}
\left|\widehat{K}(x)\right|\leq \|\varphi\|_{L^1(\R)}.
\end{equation}
The same steps as above can be used to obtain 
\begin{equation}\label{eq15}
\left|\widehat{K_\epsilon}(x)\right|\leq \|\varphi\|_{L^1(\R)}.
\end{equation}

For any $y\neq0$ and $|x|\geq 2|y|$, we have
\begin{equation}\label{eq13}
|x-y|\geq |x|-|y|\geq \frac{|x|}{2}.
\end{equation}
 Then by $(\ref{eq13})$, we get
\begin{align}
&\int_{|x|\geq 2|y|} \left|K_\epsilon(x-y)-K_\epsilon(x)\right|dy\nonumber\\
=&
\frac{1}{\pi}\int_{|x|\geq 2|y|} \left| \int_{\R} \varphi(t) \left[\frac{\sin(|a(t)|\epsilon^{-1}(x-y))}{x-y}-\frac{\sin(|a(t)|\epsilon^{-1}x)}{x}\right] dt \right| dy
\nonumber\\
\leq&
\frac{\|\varphi\|_{L^1(\R)}}{\pi}\int_{|x|\geq 2|y|}\frac{1}{|x-y|}dy+\frac{\|\varphi\|_{L^1(\R)}}{\pi}\int_{|x|\geq 2|y|}\frac{1}{|x|}dy
\nonumber\\
\leq&
\frac{\|\varphi\|_{L^1(\R)}}{\pi}\int_{|y|\leq\frac{|x|}{2}}\frac{2}{|x|}dy+\frac{\|\varphi\|_{L^1(\R)}}{\pi}\int_{|y|\leq\frac{|x|}{2}} \frac{1}{|x|}dy
\nonumber\\
=&
\frac{3}{2\pi}\|\varphi\|_{L^1(\R)}.\label{eq17}
\end{align}
By Lemma \ref{lem2} with $(\ref{eq15})$, $(\ref{eq16})$ and $(\ref{eq17})$ we obtain
\begin{equation*}
\|F_\epsilon\|_{H^1(\R)}\leq C \|f\|_{H^1(\R)},~f\in H^1(\R),
\end{equation*}
where $C$ is independent of $\epsilon$.
\end{proof}

\begin{theorem}\label{th1}
Let $\varphi$ and $a$ be as in Theorem \ref{th2}. Additionally, assume that $a$ is odd and such that $|a|$ is decreasing, positive and bijective on $(0,+\infty)$, and that for every set $E$ of zero Lebesgue measure, the set $a^{-1}(E)$ also has zero measure. Then for $\sigma>0$ and any $f\in H^1(\R)$, we have
\begin{equation}\label{eq10}
\|F_\epsilon -f\|_{H^1}\leq C K_{\sigma}(f,\epsilon)_{H^1}\to0~(\epsilon\to0^+),
\end{equation}
where $C>0$ is independent of $\epsilon$.
\end{theorem}

To prove Theorem \ref{th1}, we need a lemma as follows. 

\begin{lemma}\label{lem1}\cite[Page 179]{Lu1995}
Suppose that $\sigma>0$ and $m(x)$ is a homogeneous $H^1(\R)$ multiplier, and the family of operators $\{ M_\epsilon \}_{\epsilon>0}$ is defined by $(\ref{eq4})$. If there exists a $d>0$ such that
\begin{enumerate}
\item[(1)] $|m(x)-1|\leq C |x|^\sigma$ for $|x|\leq d$;
\item[(2)] For each $0<R<d$, we have
\begin{equation}\label{eq5}
\int_{R/2<|x|<R} \left| m'(x) \right|^2 dx \leq C R^{2\sigma-1},
\end{equation}
\end{enumerate}
then 
\begin{equation}\label{eq6}
\|M_\epsilon f-f\|_{H^1}\leq C K_\sigma (f,\epsilon)_{H^1},~ f\in H^1(\R),
\end{equation}
where $C>0$ is independent of $R$.
\end{lemma}


\begin{proof}[Proof of Theorem \ref{th1}]
By $(\ref{eq12011})$, $(\ref{eq107})$ and Theorem \ref{th2}, we obtain that $\widehat{K}$ is a homogeneous $H^1$ multiplier. By Lemma \ref{lem3} and $(\ref{eq14})$, we get
\begin{equation}\label{eq18}
\widehat{K}(0)=\int_{\R} \varphi(t) dt=1.
\end{equation}
By Lebesgue's dominated convergence theorem with $h'(x)=0$, a.e. $x\in\R$ and $a^{-1}(E)$ has zero Lebesgue measure for any zero measure set $E\subset\R$, we obtain
\begin{equation}\label{eq11281}
[\widehat{K}(x)]'=0.
\end{equation}
From this and Lagrange's mean value theorem, it follows that
\begin{equation}\left|\widehat{K}(x)-1\right|
=
\left|\widehat{K}(x)-\widehat{K}(0)\right|
=0
\leq
|x|^\sigma.\label{eq19}
\end{equation}
Therefore, by Lemma \ref{lem1} with $(\ref{eq19})$ and $(\ref{eq11281})$, we obtain
\[
\|F_\epsilon -f\|_{H^1}\leq C K_{\sigma}(f,\epsilon)_{H^1},~f\in H^1.
\]
Consequently, by Proposition \ref{pro1}, we have
\[
\|F_\epsilon -f\|_{H^1}\leq C K_{\sigma}(f,\epsilon)_{H^1}\to0~(\epsilon\to0^+), ~f\in H^1.
\]
\end{proof}

\section{Examples \label{s4}}
In the results concerning partial Hausdorff integrals, there is flexibility in the choice of \(\varphi\). This section will select some types of typical specific functions for \(\varphi\).\begin{example}\label{ex1}
Let $p>1$, $\sigma>0$, $\epsilon>0$ and $f\in L^2(\R)$. When $a(t)=\frac{1}{t}$ and $\varphi(t)=\frac{p-1}{2|t|^p}\chi_{(1,+\infty)}(|t|)$, the general Hausdorff operator $H_{\varphi,a}$ reduces to
\[
H_{\varphi,a}(f)(x)=\frac{p-1}{2}\int_{|t|>1}\frac{1}{|t|^{p+1}}f\left(\frac{x}{t}\right) dt,~x\in\R.
\]
The corresponding partial Hausdorff integrals reduces to
\[
F_\epsilon=\frac{p-1}{2\pi}\int_{|t|>1}\frac{1}{|t|^p}\int_{\R}f(x-\epsilon s)\frac{\sin\left(\frac{s}{|t|}\right)}{s}dsdt,~x\in\R.
\]
Obviously, $\varphi(t)$ and $a(t)$ satisfy the assumptions of Theorem \ref{th2} and Theorem \ref{th1}, thus for $f\in H^1(\R)$
\[
\|F_\epsilon\|\leq C_1\|f\|_{H^1}~\text{and}~\|F_\epsilon-f\|_{H^1}\leq C_2K_\sigma(f,\epsilon)\to0,~(\epsilon\to0^+),
\] 
where $C_1$, $C_2>0$ are independent of $\epsilon$.
\end{example}

\begin{example}\label{ex2}
Let $p<\frac{1}{2}$, $\sigma>0$, $\epsilon>0$ and $f\in L^2(\R)$. When $a(t)=\frac{1}{t}$ and $\varphi(t)=\frac{1-p}{2|t|^p}\chi_{(0,1)}(|t|)$, the general Hausdorff operator $H_{\varphi,a}$ reduces to
\[
H_{\varphi,a}(f)(x)=\frac{1-p}{2}\int_{|t|<1}\frac{1}{|t|^{p+1}}f\left(\frac{x}{t}\right)dt,~x\in\R.
\]
The corresponding partial Hausdorff integrals reduces to
\[
F_\epsilon=\frac{1-p}{2\pi}\int_{|t|<1}\frac{1}{|t|^p}\int_{\R}f(x-\epsilon s)\frac{\sin\left(\frac{s}{|t|}\right)}{s}dsdt,~x\in\R.
\]
Obviously, $\varphi(t)$ and $a(t)$ satisfy the assumptions of Theorem \ref{th2} and Theorem \ref{th1}, thus for $f\in H^1(\R)$
\[
\|F_\epsilon\|_{H^1}\leq C_1\|f\|_{H^1}~\text{and}~\|F_\epsilon-f\|_{H^1}\leq C_2K_\sigma(f,\epsilon)\to0,~(\epsilon\to0^+),
\] 
where $C_1$, $C_2>0$ are independent of $\epsilon$.
\end{example}

\begin{example}
In Example \ref{ex2}, if we pick \(p = 0\), the operator \(H_{\varphi,a}\) reduces to the adjoint Hardy operator \(H^*\),
\[
H^*(f)(x):=\frac{1}{2}\int_{|t|>|x|}\frac{f(t)}{|t|}dt,~x\in\R.
\]
\end{example}

\begin{example}\label{ex3}
Let $\alpha>0$, $\sigma>0$, $\epsilon>0$ and $f\in L^2(\R)$. When $a(t)=\frac{1}{t}$ and $\varphi(t)=\frac{1}{2}(1+\alpha)(1-|t|)^\alpha\chi_{(0,1)}(|t|)$, the general Hausdorff operator $H_{\varphi,a}$ reduces to Riemann-Liouville type integral, i.e.
\[
H_{\varphi,a}(f)(x)=\frac{1+\alpha}{2}\int_{|t|<1}\frac{(1-|t|)^\alpha}{|t|}f\left(\frac{x}{t}\right)dt,~x\in\R.
\]
The corresponding partial Hausdorff integrals reduces to
\[
F_\epsilon=\frac{1+\alpha}{2\pi}\int_{|t|<1}(1-|t|)^\alpha \int_{\R}f(x-\epsilon s)\frac{\sin\left(\frac{s}{|t|}\right)}{s}dsdt,~x\in\R.
\]
Obviously, $\varphi(t)$ and $a(t)$ satisfy the assumptions of Theorem \ref{th2} and Theorem \ref{th1}, thus for $f\in H^1(\R)$
\[
\|F_\epsilon\|_{H^1}\leq C_1\|f\|_{H^1}~\text{and}~\|F_\epsilon-f\|_{H^1}\leq C_2K_\sigma(f,\epsilon)\to0,~(\epsilon\to0^+),
\] 
where $C_1$, $C_2>0$ are independent of $\epsilon$.
\end{example}

\noindent{\bf Funding } This project is supported by National Natural Science Foundation of China (No.12261083) and Natural Science Foundation of Xinjiang Province (No.2024D01C40).

\noindent{\bf Data Availability} No datasets were generated or analysed during the current study.

\section*{Declarations}

\noindent{\bf Ethical Approval} Not applicable.

\noindent{\bf Competing interests} The authors declare no competing interests.

\end{document}